\documentclass[11pt]{article}
\usepackage{amsthm, amsmath, amssymb, amsfonts, url, booktabs, tikz, setspace, fancyhdr, bm}
\usepackage{amsthm}
\usepackage{mathrsfs} 
\usepackage{xcolor}
\usepackage[margin = 2.2cm]{geometry}
\usepackage{hyperref, enumerate}
\usepackage[shortlabels]{enumitem}
\usepackage[babel]{microtype}
\usepackage[english]{babel}
\usepackage[capitalise]{cleveref}
\usepackage{comment}
\usepackage{bbm}
\usepackage{csquotes}
\usepackage{mathabx}
\usepackage{tikz}
\usepackage{graphicx}
\usepackage{float}
\usepackage{thm-restate}
\usepackage{mathtools}
\usepackage{tikz}
\usepackage{graphicx}
\usepackage{float}
\usepackage{xcolor}
\usetikzlibrary{positioning}
\usetikzlibrary{shapes,arrows,arrows,positioning,fit}
\usetikzlibrary{decorations.pathreplacing}
\usetikzlibrary{decorations.pathmorphing}

\newtheorem{theorem}{Theorem}[section]

\newtheorem{lemma}[theorem]{Lemma}

\newtheorem{cor}[theorem]{Corollary}

\newtheorem{conj}[theorem]{Conjecture}

\theoremstyle{definition}
\newtheorem{prob}[theorem]{Problem}
\newtheorem{defn}[theorem]{Definition}

\newtheorem{claim}[theorem]{Claim}

\newenvironment{poc}{\begin{proof}[Proof of claim]}{\end{proof}}

\newcommand{\C}[1]{{\protect\mathcal{#1}}}

\newcommand{\ex}{\mathrm{ex}}
\usepackage{todonotes}

\newcommand{\SM}[1]{{\color{blue}{SS: #1}}}

\title{Induced even cycles in locally sparse graphs}
\begin{document}

\author{Laihao Ding\thanks{School of Mathematics and Statistics, and Key Laboratory of Nonlinear Analysis \& Applications (Ministry of Education), Central China Normal University, Wuhan, China, and Extremal Combinatorics and Probability Group (ECOPRO),  Institute for Basic Science (IBS), Daejeon, South Korea. Supported by the National Nature Science Foundation of China (11901226), the China Scholarship Council and the Institute for Basic Science (IBS-R029-C4). \textbf{Email address}: dinglaihao@ccnu.edu.cn}
\and Jun Gao\thanks{Extremal Combinatorics and Probability Group (ECOPRO), Institute for Basic Science (IBS),  Daejeon, South Korea. Supported by the Institute for Basic Science (IBS-R029-C4). \textbf{Email address}: jungao@ibs.re.kr, hongliu@ibs.re.kr}
\and Hong Liu\footnotemark[2]
\and Bingyu Luan\thanks{School of Mathematics, Shandong University, Jinan, China, and Extremal Combinatorics and Probability Group (ECOPRO),  Institute for Basic Science (IBS), Daejeon, South Korea. Supported by the China Scholarship Council, the Natural Science Foundation of China (12231018), National Key R\&D Program of China (2023YFA1009603), and the Institute for Basic Science (IBS-R029-C4). \textbf{Email address}: byluan@mail.sdu.edu.cn}
\and Shumin Sun\thanks{Mathematics Institute, University of Warwick, Coventry, UK, and Extremal Combinatorics and Probability Group (ECOPRO),  Institute for Basic Science (IBS), Daejeon, South Korea. Supported by ERC Advanced Grant 101020255 and the Institute for Basic Science (IBS-R029-C4). \textbf{Email address}: Shumin.Sun@warwick.ac.uk}
}
\date{ }
\maketitle
\begin{abstract}
     A graph $G$ is \textit{$(c,t)$-sparse} if for every pair of vertex subsets $A,B\subset V(G)$ with $|A|,|B|\geq t$, $e(A,B)\leq (1-c)|A||B|$. In this paper we prove that for every $c>0$ and integer $\ell$, there exists $C>1$ such that if an $n$-vertex graph $G$ is $(c,t)$-sparse for some $t$, and has at least $C t^{1-1/\ell}n^{1+1/\ell}$ edges, then $G$ contains an induced copy of $C_{2\ell}$. This resolves a conjecture of Fox, Nenadov and Pham.
\end{abstract}

\section{Introduction}

For a graph $H$, its \emph{Tur{\'a}n number}, denoted by $\ex(n,H)$, is the maximum number of edges in an $n$-vertex graph which does not contain a copy of $H$ as a subgraph.
The Tur{\'a}n type problem is one of the most fundamental topics in extremal combinatorics, with its origin dating back to Mantel in 1907, who solved the case when forbidding triangles. This was later generalised to all cliques by Tur{\'a}n in 1941.
More generally, the celebrated Erd\H{o}s-Stone-Simonovits theorem \cite{erdos1966limit,erdos1946structure} states that $\ex(n,H)=\left(1-\frac{1}{\chi(H)-1}+o(1)\right)\binom{n}{2}$, where $\chi(H)$ is the chromatic number of $H$. This theorem asymptotically solves the Tur{\'a}n problem for non-bipartite graphs. However, when $H$ is bipartite, our knowledge about $\ex(n,H)$ is rather limited. Even for complete bipartite graphs and even cycles, the determination of the order of the magnitudes of their Tur{\'a}n numbers turns out to be notoriously difficult. Let $C_{2\ell}$ be the cycle of length $2\ell$. While the classical upper bound $\ex(n,C_{2\ell})=O_\ell(n^{1+1/\ell})$ has been proved by Bondy and Simonovits~\cite{74BS} fifty years ago, the matching lower bounds are known to exist only for $C_4$, $C_6$ and $C_{10}$, see~\cite{66Brown,66ERS,91W}. Whether $\ex (n,C_{2\ell})=\Omega_\ell(n^{1+1/\ell})$ holds for $\ell \notin \{2,3,5\}$ has remained a long-standing open problem. For complete bipartite graphs $K_{s,t}$ with $s\le t$, the well-known K\H{o}v{\'a}ri-S{\'o}s-Tur{\'a}n theorem~\cite{kHovari1954problem} states that ${\rm ex}(n,K_{s,t})=O_{s,t}(n^{2-1/s})$. This bound is only known to be sharp when $s\in\{2,3\}$ or $t$ is sufficiently large as a function of $s$ \cite{alon1999norm,bukh2024extremal,kollar1996norm}.

In recent years, induced variants of the Tur\'an problem have gained much attention. Note that the maximum number of edges of an $n$-vertex graph is trivially $\binom{n}{2}$ when the forbidden induced graph is not a clique. So simply prohibiting induced copies of a specific graph is not very meaningful. Loh, Tait, Timmons and Zhou~\cite{loh2018induced} introduced the question of determining the maximum number of edges in an $n$-vertex graph forbidding simultaneously an induced copy of $H$ and a (not necessarily induced) copy of some graph $F$. For a fixed $H$, denote by $\C P_H$ the graph property of not containing $H$ as an induced subgraph. Let $\ex(n,\C P_H,F)$ be the maximum number of edges in an $n$-vertex member of $\C P_H$ which contains no copy of $F$. Concerning this function, a variety of papers, trying to determine $\ex(n,\C P_H,K_{t,t})$ for different graphs $H$, have appeared recently, see~\cite{24AZ,24BBCD,24GH,24HMST,22BBPRTW,04KO}. In particular, Hunter, Milojevi\'c, Sudakov and Tomon~\cite{24HMST} established the upper bound $\ex(n,\C P_{C_{2\ell}},K_{t,t})\le C_{t,\ell}n^{1+1/\ell}$ for $\ell\geq2$.

In general, we say that a graph property is \emph{hereditary} if it is closed under taking induced subgraphs. Throughout the paper, we assume that  every property $\mathcal{P}$ is non-trivial, namely it contains all edgeless graphs and misses some graph. For a hereditary property $\mathcal{P}$ and a graph $\Gamma$, let ${\rm ex}(\Gamma,\mathcal{P})$ denote the maximum number of edges of a subgraph $G\subseteq \Gamma$ that belongs to $\mathcal{P}$. For instance, if $\Gamma$ is the complete graph $K_n$ and $\mathcal{P}$ is the property of not containing $H$ as a subgraph, then ${\rm ex}(\Gamma,\mathcal{P})$ is the same as ${\rm ex}(n,H)$. Generalising Erd\H{o}s-Stone-Simonovits theorem of the classical Tur{\'a}n function, Alon, Krivelevich and Samotij~\cite{23NKS} proved that for every non-trivial hereditary property $\C P$ and the binomial random graph $G(n,p)$ with fixed $0<p<1$,
$\ex(G(n,p),\mathcal{P})=\left(1-\frac{1}{k-1}+o(1)\right)p\binom{n}{2}$ holds with high probability, where $k=k(\mathcal{P})\geq 2$ denotes the minimum chromatic number of a graph that does not belong to $\mathcal{P}$. Note that in their result, if $\mathcal{P}$ misses a bipartite graph, then ${\rm ex}(\Gamma,\mathcal{P})=o(n^2)$. 

In this paper, we focus on a larger class of graphs $\Gamma$, proposed by Fox, Nenadov and Pham \cite{fox2024largest}. A graph $G$ is \textit{$(c,t)$-sparse} if for every pair of vertex subsets $A,B\subset V(G)$ (not necessarily disjoint) with $|A|,|B|\geq t$, we have $e(A,B)\leq (1-c)|A||B|$. Here, $e(A,B)$ is the number of pairs $(a,b)\in A\times B$ such that $\{a,b\}$ is an edge in $G$. In particular, every edge lying in $A\cap B$ is double-counted. Trivially, if a graph is $(c,t)$-sparse, then so do all of its subgraphs. By averaging, it is easy to see that a graph is $K_{t,t}$-free if and only if it is $(1/t^2,t)$-sparse.
Furthermore, standard concentration inequalities infer that for $0<p<1$, with high probability, the binomial random graph $G(n,p)$ is $(c,t)$-sparse if $c<1-p$ and $t=\Omega_c(\log n)$. Hence, the investigation of $(c,t)$-sparse graphs encompasses the study of the above two induced Tur\'an variants when the host graphs are $K_{t,t}$-free graphs or dense random graphs.

Fox, Nenadov and Pham~\cite{fox2024largest} initiated the systematic study of $\ex (\Gamma, \C P_H)$ with $\Gamma$ being a $(c,t)$-sparse graph. A bipartite graph $H=(A\cup B,E)$ is \textit{$d$-bounded} if there are $1\leq \ell\leq d$ vertices in $A$ complete to $B$, and all other vertices in $A$ have degree at most $d$.  They proved that, for a $d$-bounded bipartite $H$, $\ex(\Gamma, \C P_H)\le C_{c,H}t^{1/d}n^{2-1/d}$. Also, they obtained a result for a fixed tree $T$, proving that $\ex(\Gamma, \C P_T)\le C_{c,T}tn$. In their paper, they proposed the following conjecture for even cycles.

\begin{conj}\label{conj}
    For every $c>0$ and integer $\ell$, there exists $C>1$, such that if $\Gamma$ is a $(c,t)$-sparse graph, then $\ex(\Gamma,\C P_{C_{2\ell}})\le C t^{1-1/\ell}n^{1+1/\ell}.$
\end{conj}

This conjecture can be viewed as an induced analogue of the classical Bondy-Simonovits theorem. In this paper, we confirm \Cref{conj} by proving the following theorem.

%\SM{We generally use $\ell$ instead of $l$. I have changed this in the subsequent.}

%\GJ{The introduction needs to be revised}

\begin{theorem}\label{main}
    For every $c>0$ and integer $\ell$, there exists $C>1$ such that if an $n$-vertex graph $G$ is $(c,t)$-sparse for some $t$, and has at least $C t^{1-1/\ell}n^{1+1/\ell}$ edges, then $G$ contains an induced copy of $C_{2\ell}$.
\end{theorem}

Our theorem generalises the result of Hunter, Milojevi\'c, Sudakov and Tomon~\cite[Theorem 1.4]{24HMST} on induced Tur{\'a}n problem for even cycles from $K_{t,t}$-free graphs to all $(c,t)$-sparse graphs. Our bound is best possible up to the constant $C$. In particular, the exponent of $t$ cannot be improved in general. Indeed, 
let $G$ be the disjoint union of cliques of size $(1-c)t$, then it is $(c,t)$-sparse, contains no induced $C_{2\ell}$ and has $\Omega(n^2)=\Omega(t^{1-1/\ell}n^{1+1/\ell})$ edges when $t=\Omega(n)$. 
\medskip

\noindent \textbf{Our approach.} Our proof primarily employs a concept of `good/admissible path' (see Definition \ref{good}) which helps with controlling the number of induced paths between given endpoints. The bulk of our proof is to establish the key lemma, Lemma~\ref{lem: main}. Roughly speaking, \Cref{lem: main} states that if an almost regular and locally sparse graph has many `minimally congested' (bad admissible) induced $s$-paths with $3\le s\le \ell$ ($\mathcal{F}'$ in~\Cref{cl:paths}) starting from some same vertex $x$, then one can utilise them to build an induced even cycle. If we do not insist on the embedding to be induced, the proof of Lemma~\ref{lem: main} would be quite straightforward. However, the inducedness requirement makes the embedding procedure much more delicate and involved. We first show that there are many internally disjoint $s$-paths in $\mathcal{F}'$ with the same endpoints ($\{\mathcal{P}_y\}_{y\in Y}$ in Claim \ref{cl:max-disj-path}) and there are many edges that can be served as the last edge for these $s$-paths in $\mathcal{F}'$ ($H$ in Claim \ref{lm:H-edge}). We carry out the embedding of an induced $C_{2\ell}$ in two stages. Firstly, we embed in $H$, these last edges of the $s$-paths, an induced $(2\ell-2s)$-path $P$ between some $z_1,z_m\in Y$ (see Section \ref{Gamma} and Claim \ref{lm:2l-2s}). We then utilise the local sparseness of $G$ to find induced $s$-paths from $x$ to $z_1$ and to $z_m$ to form an induced $C_{2\ell}$ with $P$ (Claim \ref{lm:glue}). This embedding strategy breaks down when $s=2$ as $x$ necessarily is adjacent to some internal vertices of $P$. We instead use the dependent random choice method to construct an induced $C_{2\ell}$ in the last edges of these bad admissible $2$-paths. In order to proceed the two strategies above, a key ingredient is Lemma \ref{coro2}: it bootstraps local sparseness of a graph in \emph{any} hereditary family missing a bipartite graph, which we believe will be useful for other induced bipartite Tur\'an problems. 
\medskip

\noindent \textbf{Organisation.} This paper is organised as follows. Section 2 introduces some auxiliary results. In Section 3, we prove Theorem~\ref{main} assuming the key lemma,~\Cref{lem: main}. The proof of~\Cref{lem: main} is split into two cases, given in Section 4 and Section 5 respectively. Concluding remarks are given in Section 6.

\section{Preliminaries}
\noindent $\mathbf{Notation.}$ Let $G$ be a graph. We write $v(G)$ and $e(G)$ for the order and the number of edges of $G$ respectively. By $\delta(G)$, $d(G)$ and $\Delta(G)$, we denote the minimum, average and maximum degree of $G$ respectively. For vertices $u$, $v\in V(G)$ and $A\subseteq V(G)$, denote by $N_G(u)$ the neighbourhood of $u$ and $N_A(u)=N_G(u)\cap A$, and write $N_G(u,v)=N_G(u)\cap N_G(v)$ for the \textit{common neighbourhood} of $u,v$ and $N_A(u,v)=N_G(u,v)\cap A$. Define $d_G(u):=|N_G(u)|$ as the \textit{degree} of $u$ in $G$, and similarly, we define $d_A(u)$. For $W\subseteq V(G)$, the induced subgraph of $G$ on $W$ is denoted by $G[W]$. For a collection $\mathcal{G}$ of graphs, we write $V(\mathcal{G})=\bigcup_{H\in \mathcal{G}}V(H)$. 
The length of a path is the number of edges of the path, and a path is called an \textit{$s$-path} if it is of length $s$. We call a path with two vertices $u,v$ as endpoints a \textit{$(u,v)$-path}. 
%For a path $P$ and two vertices $x,y\in V(P)$, denote by $P[x,y]$ the $(x,y)$-subpath of $P$. 
The \emph{internal set} of a $(u,v)$-path $P$ is the set $V(P)\setminus\{u,v\}$, and two paths are \emph{internally vertex-disjoint} if their internal sets are disjoint. We write  $x\ll y\ll z$ to mean that we can choose the positive constants  $x, y, z$ from right to left. More precisely, there are increasing functions
$f$ and $g$ such that, given $z$, whenever we choose some $y\le  f(z)$ and $x\le g(y)$, the subsequent
statement holds. Hierarchies of other lengths are defined similarly.

We need a result of Fox, Nenadov and Pham \cite{fox2024largest} for bounded bipartite graphs.

\begin{cor}[\cite{fox2024largest}]\label{thm-original}
    Let $H$ be a bipartite graph with all vertices in one part having degree at most $d$. If $\Gamma$ is a $(c,t)$-sparse $n$-vertex graph, for some $c>0$ and $t$, then there exists $C=C(c,H)>0$ such that any subgraph $G\subseteq \Gamma$ with at least $Ct^{1/d}n^{2-1/d}$ edges contains a copy of $H$ which is induced in $\Gamma$.
\end{cor}

A graph $G$ is $K$-$almost$-$regular$ if $\Delta(G)\leq K\delta(G)$. The following lemma allows us to consider only almost regular graphs.

\begin{lemma}[\cite{24HMST}]\label{almostregular}
    Let $\alpha>0$ be a fixed real number and let $K:=2^{3/\alpha +4}$. Furthermore, let $G$ be an $n$-vertex graph with at least $Cn^{1+\alpha}$ edges. Then, there exists a $K$-almost-regular induced subgraph $H\subseteq G$ on $m$ vertices with at least $\frac{C}{4}m^{1+\alpha}$ edges.
\end{lemma}

To construct an induced copy of $H$, finding independent sets is important for preserving non-adjacencies during the embedding. The following lemma from \cite{fox2024largest} shows that the $(c,t)$-sparseness is useful in constructing independent sets.

\begin{lemma}[\cite{fox2024largest}]\label{independent}
    Let $\Gamma$ be a $(c,t)$-sparse graph with $0<c\leq 1/2$. Let $k\geq 2$ and $V_1,\dots,V_k$ be vertex subsets of $\Gamma$, not necessarily disjoint, each of size at least $m\geq k^2c^{-k^2}t$. A uniformly random $k$-tuple in $V_1\times \dots\times V_k$ forms an independent set of $k$ distinct elements with probability at least $c^{k^2}$.
\end{lemma}

\section{Proof of Theorem~\ref{main}}
\subsection{Main lemmas}
The first lemma we need states that the local sparseness of a graph can be bootstrapped if it is induced $H$-free for \emph{any} bipartite $H$.

\begin{lemma}\label{coro2}
    Let $H$ be a bipartite graph. For any $c,\varepsilon >0$, there exists $\beta_0 =\beta(c,H,\varepsilon)$ such that for any $\beta \ge \beta_0$ and every $(c,t)$-sparse graph $G$, if $G$ does not contain an induced copy of $H$, then it is also $(1-\varepsilon,\beta t)$-sparse.
\end{lemma}

\begin{proof}
    Clearly, $H$ is a bipartite graph with all vertices in one part having degree at most $v(H)$.
    Let $C=C(c,H)$ be the constant such that the conclusion of \Cref{thm-original} holds and set $\beta_0 =(8C/\varepsilon)^{v(H)}$. Suppose to the contrary that there exists a $(c,t)$-sparse graph $G$ which is induced $H$-free   but not $(1-\varepsilon,\beta t)$-sparse for some $\beta \geq \beta_0$. Then there exist vertex subsets $A,B\subseteq V(G)$ such that $|B|\geq |A|\geq\beta t$ and $e(A,B)\geq \varepsilon |A||B|$. By averaging, there is a subset $B'\subseteq B$ such that $|B'|=|A|$ and $e(A,B')\geq\varepsilon |A|^2$. Then, the induced subgraph $G'=G[A\cup B']$ satisfies that $|G'|\leq 2|A|$ and 
    \begin{align*}
        e(G')&\geq \frac{e(A,B')}{2}\ge \frac{\varepsilon |A|^2}{2}\geq\frac{\varepsilon}{2}(\beta t)^{1/v(H)}|A|^{2-1/v(H)}\\
        &\geq \frac{\varepsilon \beta ^{1/v(H)}}{8} t^{1/v(H)}|G'|^{2-1/v(H)}\ge Ct^{1/v(H)}|G'|^{2-1/v(H)}.
    \end{align*}
    Applying Corollary \ref{thm-original} with $(c,t,d,H,\Gamma)_{\ref{thm-original}}=(c,t,v(H),H,G')$, we can find an induced copy of $H$ in $G'$, which is also an induced copy of $H$ in $G$, and thus we get a contradiction. 
\end{proof}

The second lemma harnesses this strengthened local sparseness to find many induced paths in almost regular graphs.

\begin{lemma}\label{pathnumber}
Let $C,\beta ,\varepsilon, K, \ell$ be positive reals such that  $1/C\ll 1/\beta\ll \varepsilon \ll 1/K, 1/\ell.$ Let $G$ be a $K$-almost regular and $(1-\varepsilon, \beta t)$-sparse graph with average degree $d \ge 2Ct$. Then the number of induced $\ell$-paths in $G$ is at least $\frac{n d^{\ell}}{2^{\ell+1}K^{\ell}}$.
\end{lemma}

\begin{proof}
    For every $u\in V(G)$, set 
    $$X_u:=\{v~\big | ~ v\ne u, |N_G(u,v)|>\varepsilon Kd\}.$$ 
    We first show that $|X_u|<\beta t$ for every $u$. Otherwise, let $u$ be any vertex with $|X_u|\geq \beta t$. Since $G$ is $K$-almost-regular, we obtain that $d/K\leq |N_G(u)|\leq Kd$.
    Since $C\gg \beta \gg \ell, K$ and $d\ge 2Ct$, we get that $|N_G(u)| \ge \beta t$. 
    Therefore,
    $e(N_G(u),X_u)\geq \varepsilon Kd\cdot |X_u|\geq \varepsilon |N_G(u)||X_u|$, contradicting that $G$ is $(1-\varepsilon,\beta t)$-sparse. 
    
    To prove the claim, it suffices to show that for every  $u\in V(G)$ and $1\le s \le \ell$, there are at least $\frac{d^{s}}{2^{s}K^{s}}$ induced $s$-paths  $P=v_0v_1v_2\ldots v_{s}$ which start from $u = v_0$ satisfying that $v_{s} \notin X_{v_{i}}$ for $0\leq i \leq s-1$. If so, then the number of induced $\ell$-paths in $G$ is at least $\frac{1}{2}\cdot n\cdot \frac{d^{\ell}}{2^{\ell}K^{\ell}}=\frac{n d^{\ell}}{2^{\ell+1}K^{\ell}}$ as desired.

    We prove it by induction on $s$.
    When $s=1$, this statement clearly holds by considering all the edges $uv_1$ with $v_1\in N_G(u)\setminus X_u$ as $|N_G(u)\setminus X_u| \ge d/K-\beta t \ge d/2K$.
    Now suppose that $s \geq 2$ and this statement holds for $s-1$. Let $Q=v_0v_1v_2\ldots v_{s-1}$ be any induced $(s-1)$-path satisfying that $v_{s-1} \notin  X_{v_{i}}$ for $0\leq i \leq s-2$. Let $A_{v_{s}}=N_G(v_{s-1})\setminus\left(\bigcup_{i=0}^{s-2}N_G(v_i)\cup V(Q)\right)$. 
    For $0\leq i \leq s-2$, since $v_{s-1} \notin X_{v_{i}}$, we have $|N_G(v_i)\cap N_G(v_{s-1})| \le \varepsilon Kd$. Then $|A_{v_{s}}|\geq d/K-(s-1) \varepsilon Kd -s\geq 3d/4K$ and for each $v\in A_{v_{s}}$, the $s$-path $v_0v_1v_2\ldots v_{s-1}v$ is induced. Now let $Y_{v_{s}}=A_{v_s}\setminus\left(\bigcup_{i=0}^{s-1}X_{v_i}\right)$. Then we have $|Y_{v_{s}}| \ge  |A_{v_{s}}| - s\beta t \ge 3d/4K - \ell \beta t \ge d/2K$.
    And for any $v\in Y_{v_{s}}$, $v_0v_1v_2\ldots v_{s-1}v$ is an induced path with $v\notin X_{v_i}$ for $0\leq i \leq s-1$.
    So there are at least $\frac{d^{s-1}}{2^{s-1}K^{s-1}}\cdot \frac{d}{2K}=\frac{d^{s}}{2^{s}K^{s}}$ induced $s$-paths  $P=v_0v_1v_2\ldots v_{s}$ which start from $u = v_0$ satisfying that $v_{s} \notin X_{v_{i}}$ for $0\leq i \leq s-1$. 
\end{proof}

We next introduce the key lemma of this paper, and the proof of it will be divided into two cases and presented in the subsequent two sections. Before the formal statement, we need a notion of admissible and good paths, which was first introduced in \cite{21CJL} and is recursively defined as follows.

\begin{defn}\label{good}
    Any path of length one is both admissible and good. For $s\geq 2$, an induced $s$-path is \textit{admissible} if any $(s-1)$-subpath of it is good, and it is \textit{good} if it is admissible and shares the same endpoints with at most $(Lt)^{s-1}$ admissible $s$-paths\footnote{To be precise, admissible/good paths are defined with respect to parameters $L$ and $t$. As we will use the same $L$ and $t$ throughout the proof, we omit mentioning them.}. Finally, an induced path is \textit{bad} if it is not good.
\end{defn}

Note that by the definition, all the induced 2-paths are admissible, and for every $i\in [s-1]$, any $i$-subpath of an admissible $s$-path is good. Our key lemma is as follows.

\begin{lemma}\label{lem: main}
    Let $C,L,\beta ,\varepsilon, K, \ell$ be positive reals such that  $1/C\ll 1/L\ll 1/\beta\ll \varepsilon \ll 1/K, 1/\ell.$ Let $G$ be a $K$-almost regular and $(1-\varepsilon, \beta t)$-sparse graph with average degree $d \ge 2Ct$. If for some $s$ with $2\le s\le \ell$, the number of bad admissible $s$-paths in $G$ is at least $$\frac{1}{2}(1/4K)^{2\ell-s}nd^s,$$
then $G$ contains an induced copy of $C_{2\ell}$.
\end{lemma}

\subsection{Putting things together}
\begin{proof}[Proof of Theorem~\ref{main}]
Choose parameters satisfying the following hierarchy:
\[
0<\frac{1}{C}\ll \frac{1}{L} \ll \frac{1}{\beta}\ll \varepsilon \ll c,\frac{1}{\ell}.
\]
Suppose to the contrary that $G$ is an $n$-vertex $(c,t)$-sparse graph with $e(G) \ge C t^{1-1/\ell}n^{1+1/\ell}$ and no induced copy of $C_{2\ell}$.

By Lemma~\ref{almostregular}, we can assume that $G$ is $K$-almost-regular, where $K= 2^{3\ell +4}$. 
By \Cref{coro2},  we derive that $G$ is $(1-\varepsilon,\beta t)$-sparse.
Let $d =d(G)$ be the average degree of $G$. 
Since $n^2>e(G) \ge  C t^{1-1/\ell}n^{1+1/\ell}$ and $C \gg 1$, we have $n\ge t$.
Then $d\geq 2Ct^{1-1/\ell}n^{1/\ell}\geq 2Ct$. By~\Cref{lem: main}, we derive that for any $2\le s\le \ell$, the number of bad admissible $s$-paths is at most $\frac{1}{2}(1/4K)^{2\ell-s}nd^s$.

We claim that for any $1\le s \le \ell$, the number of bad induced $s$-paths is at most $(1/4K)^{2\ell-s}nd^s$.
We prove it by induction on $s$.
When $s=1$, since any path of length one is both admissible and good, we are done.
Now suppose that $s \geq 2$ and this statement holds for $s-1$.
Then by the inductive hypothesis, we know that the number of bad induced $(s-1)$-paths is at most $ (1/4K)^{2\ell-(s-1)}nd^{s-1}$.
Since any induced $s$-path which is not admissible contains a bad induced $(s-1)$-subpath and recall that $\Delta(G)\le Kd$,
we derive that the number of non-admissible induced $s$-paths is at most $2\cdot Kd \cdot \big( (1/4K)^{2\ell-(s-1)}nd^{s-1}\big) = \frac{1}{2}(1/4K)^{2\ell-s}nd^s$.
By the definition of good paths, if an induced path is bad, then either it is non-admissible or it is admissible but bad.
Together with~\Cref{lem: main}, we derive that the number of bad induced $s$-paths is at most $(1/4K)^{2\ell-s}nd^s$ as claimed.

Specially, the number of bad induced $\ell$-paths is at most $(1/4K)^{\ell}nd^\ell$.
Since the number of induced $\ell$-paths is at least $\frac{n d^{\ell}}{2^{\ell+1}K^{\ell}}$ by~\Cref{pathnumber}, we derive that the number of good induced $\ell$-paths is at least 
\begin{align*}
    \frac{n d^{\ell}}{2^{\ell+1}K^{\ell}}  -(1/4K)^{\ell}nd^\ell \ge (1/4K)^{\ell}nd^\ell \geq  (C/4K)^{\ell} t^{\ell -1} n^2 > n^2 (Lt)^{\ell -1}.
\end{align*}
On the other hand, by definition, the number of good induced  $\ell$-paths is at most $n^2 (Lt)^{\ell -1}$, a contradiction. This finishes the proof of Theorem~\ref{main}.
\end{proof}

\section{Proof of~\Cref{lem: main} for $s =2$}

Take constants $M,\varepsilon_0$ such that $$\frac{1}{L} \ll \frac{1}{M} \ll \frac{1}{\beta} \ll \varepsilon_0 \ll \varepsilon.$$
Suppose that there are at least $\frac{1}{2}(1/4K)^{2\ell-2}nd^2$ bad admissible $2$-paths. Then, by pigeonhole, there exists a vertex $x$ such that the number of bad admissible $2$-paths starting from $x$ is at least $\frac{1}{2}(1/4K)^{2\ell-2}d^2$. Let $A$ and $B$ be the vertex sets of all the second and last vertices on such paths respectively, so we have $e(A,B)\geq \frac{1}{2}(1/4K)^{2\ell-2}d^2$. 
Since each such $2$-path is induced, we derive that $A$ and $B$ are disjoint.
Since $G$ is $K$-almost regular, we have $|A| \le d_G(x) \le Kd.$
By the definition of bad paths, $d_A(u)\geq Lt$ for every vertex $u\in B$, which implies that $e(A,B) \ge Lt|B|$.

Pick one vertex $u\in B$ uniformly at random. Then $\mathbb{E}\left[d_A(u)\right]=e(A,B)/|B|$. We say that a pair $(a_1,a_2)\in N_A(u)\times N_A(u) \subseteq A\times A$ is \emph{light} if $|N_B(a_1,a_2)|<Mt$, and is \emph{heavy} otherwise. Let $Y$ be the number of all light pairs in $N_A(u)\times N_A(u)$, then we have $\mathbb{E}[Y]\leq |A|^2\cdot M t/|B|$. As $L \gg M \gg 1/\varepsilon_0 \gg \ell$, we have
\begin{align*}
    \mathbb{E}\left[\varepsilon_0d^2_A(u)-Y\right] &\ge \varepsilon_0 \left(\mathbb{E}[d_A(u)]\right)^2- \mathbb{E}[Y]\geq\frac{1}{|B|}\left(\varepsilon_0\cdot \frac{e^2(A,B)}{|B|}-|A|^2\cdot M t\right)\\
    &\ge \frac{1}{|B|}\left(\varepsilon_0\cdot Lt \cdot \frac{1}{2}(1/4K)^{2\ell-2}d^2  -K^2d^2\cdot M t\right)>0.
\end{align*}
So there exists a vertex $u\in B$ such that the number of light pairs in $N_A(u)$ is at most $\varepsilon_0d^2_A(u)$. Fix such a vertex $u$. 

Uniformly at random pick $\ell$ vertices $x_1,x_2,\dots, x_\ell$ from $N_A(u)$. Define event
\[
\mathcal{E}=\Big\{\mbox{$X=\{x_1,x_2,\dots, x_\ell\}$ is an independent set consisting of $\ell$ distinct vertices}\Big\}.
\]
Since a $(1-\varepsilon,\beta t)$-sparse graph is also $(1/2, \beta t)$-sparse and
$d_A(u) \ge Lt  \ge \ell^2 (\frac{1}{2})^{-\ell^2} \beta t$, Lemma~\ref{independent} implies that  $\mathbb{P}[\mathcal{E}] \ge (1/2)^{\ell^2}$.

For any $1\le i<j \le \ell$, define events
\[
\mathcal{E}_{i,j}=\Big\{\mbox{$(x_i,x_j)$ is a heavy pair, i.e., $|N_B(x_i,x_j)| \ge Mt$}\Big\},
\]
and
\[
\mathcal{E}'_{i,j}=\Big\{\mbox{$\left|N_{i,j}\right|\geq Mt/\ell$, where $N_{i,j}=N_B(x_i,x_j)\setminus\left(\bigcup_{k\in [\ell]\setminus\{i,j\}}N_B(x_k)\right)$}\Big\}.
\]
We next bound the probability $\mathbb{P}[\mathcal{E}'_{i,j} \cap \mathcal{E}_{i,j}]$.
Fix $1\le i<j \le \ell$, $k\in[\ell]\setminus\{ i,j\}$, and suppose that $(x_i,x_j)$ is a heavy pair. Let 
\[Z:=\big\{v\in N_A(u)\backslash\{x_i,x_j\}~\big | ~ |N_B(v)\cap N_B(x_i,x_j)|\geq |N_B(x_i,x_j)|/\ell\big\}.\] 
Note that $|N_B(x_i,x_j)| \ge Mt \ge \beta t$. If $|Z|>\beta t$, then $e(Z,N_B(x_i,x_j)) \ge |Z|\cdot|N_B(x_i,x_j)|/\ell > \varepsilon |N_B(x_i,x_j)|\cdot|Z|$, a contradiction to that $G$ is $(1-\varepsilon, \beta t)$-sparse. So  $|Z|\le \beta t$.
If $x_k \notin Z$ for any $k\in[\ell]\setminus \{i,j\}$, then 
$$\left|N_{i,j}\right|\geq\left(1-\frac{\ell-2}{\ell}\right)|N_B(x_i,x_j)| > Mt/\ell .$$
Therefore, 
\begin{align*}
    \mathbb{P}[\mathcal{E}'_{i,j} ~\big |~ \mathcal{E}_{i,j}] 
    &\ge \mathbb{P}[x_k \notin Z \text{ for any } k\in[\ell]\setminus\{i,j\} ~\big |~ \mathcal{E}_{i,j}]\\ 
    &\ge  1-  ( \ell-2)\frac{ \beta t}{d_A(u)} \ge  1-  \frac{ \ell \beta}{L},
\end{align*}
where the last inequality follows from that $d_A(u)\geq Lt$.
Since the number of light pairs in $N_A(u)$ is at most $\varepsilon_0 d^2_A(u)$, we obtain that $\mathbb{P}[\mathcal{E}_{i,j}] \ge 1-\varepsilon_0$. Therefore, we derive that
\begin{align*}
    \mathbb{P}[\mathcal{E}'_{i,j}\cap \mathcal{E}_{i,j}] = \mathbb{P}[\mathcal{E}_{i,j}]\cdot \mathbb{P}[\mathcal{E}'_{i,j}~\big |~ \mathcal{E}_{i,j}]  \ge 1-\varepsilon_0- \frac{ \ell \beta}{L}.
\end{align*}

Let $\mathcal{D}$ be the event that $\mathcal{E}$, $\mathcal{E}_{i,j}$ and $\mathcal{E}'_{i,j}$ occur for all $1\leq i<j\leq \ell$.
By the union bound, we obtain that 
\begin{align*}
    \mathbb{P}[\mathcal{D}] \ge 1- (1- (1/2)^{\ell^2}) - \binom{\ell}{2} \left(\varepsilon_0+ \frac{ \ell \beta}{L}\right) >0.
\end{align*}
Thus, we can fix a choice of $X$ such that $\mathcal{D}$ occurs. Then $N_{i,j} \ge Mt/\ell > \ell^2 (\frac{1}{2})^{\ell^2}\beta t$ for any $1\le i<j\le \ell$.
By Lemma~\ref{independent}, with positive probability, there is an $\ell$-tuple $X'\in N_{1,2}\times N_{2,3}\times \dots \times N_{\ell-1,\ell}\times N_{\ell,1}$ which forms an independent set of $\ell$ distinct elements. Note that $G[X\cup X']$ is an induced copy of $C_{2\ell}$, a contradiction. This completes the proof of~\Cref{lem: main} for $s=2$.

\section{Proof of Lemma \ref{lem: main} for $s\geq3$}
We take parameter $\gamma$ such that
\[
1/C\ll 1/L\ll 1/\gamma\ll 1/\beta\ll \varepsilon \ll 1/K, 1/s, 1/\ell.
\]
Fix $s$ with $3\leq s\leq \ell$, and set $g:=\frac{1}{2}(1/4K)^{2\ell-s}$. Suppose that the number of bad admissible $s$-paths is at least $gnd^s$. Then, by pigeonhole principle, there exists a vertex $x$ such that there are at least $gd^s$ bad admissible $s$-paths starting at $x$. Denote by $\C F$ the family of these paths.

Now we greedily remove some paths from $\C F$ in the following process. Suppose $\mathcal{F}^*\subseteq\mathcal{F}$ is the current family after removing some paths from $\mathcal{F}$. If there exists an induced $s$-path $P=xx_1\cdots x_{s-1}x_s$ in $\C F^*$ satisfying that 
\begin{enumerate}[label=(\roman*)]
    \item the number of induced $s$-paths $P'\in \C F^*$ with $P'=xx_1\cdots x_{s-1}x'_s$ (namely $P'$ shares the first $s$ vertices with $P$) is at most $\frac{gd}{4K^{s-1}}$, or

    \item the number of induced $s$-paths $P'\in \C F^*$ with $P'=xx'_1\cdots x'_{s-1}x_s$ (namely $P'$ shares the same endpoints with $P$) is at most $\frac{(Lt)^{s-1}}{2}$,
\end{enumerate}
then we remove $P$ along with all associated $P'$ satisfying (i) or (ii). The procedure terminates until no such $P$ exists, and we denote by $\C F'\subseteq \C F$ the final resulting family of induced $s$-paths. 
\begin{claim}\label{cl:paths}
 $\C F'\neq \varnothing.$
\end{claim}

\begin{poc}
Let us calculate the number of paths removed because of (i) and (ii) respectively. We claim that the number of removed paths via (i) is at most
    \begin{equation}\label{eq:i}
    (Kd)^{s-1}\cdot \frac{gd}{4K^{s-1}}= \frac{gd^s}{4}\le \frac{|\C F|}{4}.
    \end{equation}
Indeed, during each deletion via (i), we actually remove a bunch of paths sharing the same subpath of length $s-1$ starting from $x$. Since $G$ is almost $K$-regular, $(Kd)^{s-1}$ is a trivial upper bound for the number of $(s-1)$-paths starting from $x$, which in turn gives an upper bound for the number of subpath options $xx_1\cdots x_{s-1}$. For every subpath $xx_1\cdots x_{s-1}$, we remove at most $\frac{gd}{4K^{s-1}}$ $s$-paths $xx_1\cdots x_{s-1}x'_s$ from $\C F$ according to (i). Hence, (\ref{eq:i}) follows.

We next prove that the number of removed paths via (ii) is at most $|\C F|/2$. Let $W$ be the union of endpoints of all paths in $\C F$ with $x$ excluded. For any vertex $y\in W$, denote by $c_{\C F}(x,y)$ the number of $(x,y)$-paths in $\C F$. Since every $P\in \C F$ is admissible and bad, there are at least $(Lt)^{s-1}$ paths $P'$ in $\C F$ (which are all admissible and bad) sharing the same endpoints with $P$. Thus, $c_{\C F}(x,y)\ge (Lt)^{s-1}$ for every $y\in W$. 
Since we remove at most $(Lt)^{s-1}/2$ paths sharing the same endpoints from $\C F$ in each deletion via (ii), the number of removed paths via (ii) is at most
    $$\sum_{y\in W}\frac{(Lt)^{s-1}}{2}\le \sum_{y\in W}\frac{c_{\C F}(x,y)}{2}=\frac{|\C F|}{2},$$
where the last equality follows from that $|\C F|=\sum_{y\in W}c_{\C F}(x,y)$.

Thus, we remove in total at most $3|\C F|/4$ paths from $\C F$ and the claim holds.
\end{poc}

Let $Y$ be the set of all endpoints of paths in $\C F'$ excluding $x$. For any vertex $y\in Y$, let $\C P_y\subseteq \C F'$ be a maximal set of internally vertex-disjoint induced $(x,y)$-paths. Then we have the following claim.

\begin{claim}\label{cl:max-disj-path}
 For each $y\in Y$, $|\C P_y|\ge \gamma t.$
\end{claim}

\begin{poc}
Let $y$ be an arbitrary vertex in $Y$. For any two vertices $v,w$ and integer $1\leq s'\leq s-1$, let $\C P_y^{s'}(v,w)$ be the collection of $(v,w)$-paths of length $s'$ each of which is a subpath of some $(x,y)$-path in $\C F'$, and let $\C P_y^{s'}(v)\subseteq \C F'$ be the collection of $(x,y)$-paths each of which contains a path from $\C P_y^{s'}(x,v)$ as a subpath.

    By the definition of $\C F'$ (in particular (ii)), there are at least $(Lt)^{s-1}/2$ $(x,y)$-paths in $\C F'$. By the maximality of $\C P_y$, for any $(x,y)$-path $P$ in $\C F'$, there exists an $(x,y)$-path $P'\in \C P_y$ ($P'$ and $P$ could be the same path) such that $P'$ internally intersects with $P$. %Let $x_0$ be the intersecting vertex closest to $x$ (in both $P$ and $P'$). Then $P[x,x_0]$ and $P'[x,x_0]$ are internally vertex-disjoint, since otherwise we would obtain an intersecting vertex of $P$ and $P'$ which is closer to $x$. 
    Therefore, by the pigeonhole principle, there exists a vertex $v\in V(\C P_y)\setminus \{x,y\}$ such that 
    \begin{equation}\label{eq:py-lb}
    |\C P_y^{s'}(v)|\geq\frac{1}{(s-1)^2|\C P_y|}\cdot \frac{(Lt)^{s-1}}{2}
    \end{equation}
    for some $1\leq s'\leq s-1$. 
Recall that each path in $\C P_y$ is admissible. Therefore, every path in  $\C P_y^{s'}(x,v)$ or  $\C P_y^{s-s'}(v,y)$ is good, which means that 
$$|\C P_y^{s'}(x,v)|\le (Lt)^{s'-1} \text{~and~} |\C P_y^{s-s'}(v,y)|\le (Lt)^{s-s'-1}.$$ 
Then we have 
\begin{equation}\label{eq:py-ub}
|\C P_y^{s'}(v)| \le |\C P_y^{s'}(x,v)|\cdot |\C P_y^{s-s'}(v,y)| \le (Lt)^{s-2}.
\end{equation}
Hence, by~(\ref{eq:py-lb}) and~(\ref{eq:py-ub}), $|\C P_y|\ge \frac{(Lt)^{s-1}}{(Lt)^{s-2}}\cdot \frac{1}{2(s-1)^2}\ge \frac{Lt}{2(s-1)^2}\ge \gamma t$ as desired.
\end{poc}

By removing some paths if necessary, we will assume that $\C P_y$ is a (not necessarily maximal) set of exactly $\gamma t$ internally vertex-disjoint induced $(x,y)$-paths in the subsequent discussion.

Next we shall construct two auxiliary graphs $H$ and $\Gamma$ on the same vertex set with $H\subseteq \Gamma$, and utilize them to find an induced $(y_1,y_2)$-path $P$ of length $2(\ell-s)$ with $y_1,y_2\in Y$. Then we will complete the proof by finding two internally vertex-disjoint paths $P_1\in \mathcal{P}_{y_1}$ and $P_2\in \mathcal{P}_{y_2}$ to form an induced $C_{2\ell}$ with $P$ in $G$. For this purpose, we need the following theorem due to Fox, Nenadov and Pham \cite{fox2024largest}.

\begin{theorem}[\cite{fox2024largest}]\label{thm:fnp}
    For any tree $T$ and $c'>0$, there exists $C'>1$ such that the following holds. Given a $(c',t')$-sparse graph $\Gamma$ on $n$ vertices, for some $t'$, every subgraph $H\subseteq \Gamma$ with at least $C't'n$ edges contains a copy of $T$ which is induced in $\Gamma$.
\end{theorem}

\subsection{The auxiliary graph $H$}
Given an $s$-path $P=xx_1\cdots x_{s}$ in $\C F'$, we call $x_i$ the \textit{ $i$-th vertex} of $P$. Denote by $A'$ and $B'$ the set of all the $(s-1)$-th and the $s$-th vertices of paths in $\C F'$ respectively. Note that it is possible that some vertex belongs to $A'$ and $B'$ simultaneously. Now we define a bipartite graph $H$ on vertex set $A\cup B$ with $A:=\{v_A~|~v\in A'\}$ and $B:=\{v_B~|~v\in B'\}$ as follows. For any two vertices $u_A\in A$ and $v_B\in B$, we add $u_Av_B$ into the edge set $E(H)$ if there is an $s$-path $P\in \C F'$ such that $u$ is the $(s-1)$-th vertex of $P$ and $v$ is the $s$-th vertex of $P$. For simplicity, for each $v\in A\cup B$ and $U\subseteq A\cup B$, we use $\varphi(v)$ to denote the corresponding vertex of $v$ in $A'\cup B'$ and set $\varphi(U)=\{\varphi(u)~:~u\in U\}$ in the following. 

\begin{claim}\label{lm:H-edge}
    $e(H)\ge \frac{\gamma t v(H)}{2}$.
\end{claim}

\begin{poc}
It suffices to show that $\delta(H)\ge \gamma t$. For any vertex $a\in A$, recall that $\varphi(a)$ is the $(s-1)$-th vertex of some path in $\C F'$. By the definition of $\C F'$, for every path $P\in \C F'$, there are at least $\frac{gd}{4K^{s-1}}$ paths $P'\in \C F'$ (including $P$) sharing the first $s$ vertices with $P$. Thus, by the definition of $H$, $d_H(a)\geq \frac{gd}{4K^{s-1}}\ge \frac{g\cdot 2Ct}{4K^{s-1}}\ge\gamma t$. For each $b\in B$, $\varphi(b)$ is an endpoint of some path in $\C F'$. By~\Cref{cl:max-disj-path}, there are $\gamma t$ internally vertex-disjoint $(x,\varphi(b))$-paths in $\C F'$ and so $d_{H}(b)\ge \gamma t$. 
\end{poc}

\subsection{The auxiliary graph $\Gamma$}\label{Gamma}
Let $\varepsilon'=(s-1)\varepsilon$. The auxiliary graph $\Gamma$ is defined on the vertex set $V(\Gamma):=A\cup B$, and with the edge set $E(\Gamma)$ consisting of the following four types of edges.

\begin{enumerate}[label=(\arabic*)]
    \item For $a\in A$ and $b\in B$, if $\varphi(a)=\varphi(b)$, then we include $ab$ as a \emph{type-1} edge of $\Gamma$.

    \item For $a\in A$ and $b\in B$, if $\varphi(a)\neq \varphi(b)$ and there are at least $\varepsilon' \gamma t$ $s$-paths $P$ in $\C P_{\varphi(b)}$ containing some vertex $v\in V(P)\setminus \{x,\varphi(b)\}$ such that $v\varphi(a)\in E(G)$, then we include $ab$ as a \emph{type-2} edge of $\Gamma$.

    \item For $b_1, b_2\in B$, if for some $1\leq i\leq2$, there are at least $\varepsilon' \gamma t$ $s$-paths $P$ in $\C P_{\varphi(b_i)}$ containing some vertex $v\in V(P)\setminus \{x,\varphi(b_i)\}$ such that $v\varphi(b_{3-i})\in E(G)$, then we include $b_1b_2$ as a \emph{type-3} edge of $\Gamma$.
    
    %For $b, b'\in B$, if there are at least $\varepsilon' \gamma t$ $s$-paths $P$ in $\C P_{\varphi(b)}$ containing some vertex $v\in V(P)\setminus \{x,\varphi(b)\}$ such that $v\varphi(b')\in E(G)$, or there are at least $\varepsilon' \gamma t$ $s$-paths $P$ in $\C P_{\varphi(b')}$ containing some vertex $v\in V(P)\setminus \{x,\varphi(b')\}$ such that $v\varphi(b)\in E(G)$, then we include $bb'$ as a \emph{type-3} edge of $\Gamma$.

    \item For $x,y\in A\cup B$, if $\varphi(x)\varphi(y)\in E(G)$, then we include $xy$ as a \emph{type-4} edge of $\Gamma$.
\end{enumerate}

\noindent For $i\in [4]$, denote by $E_i$ the set of all type-$i$ edges. Note that it is possible that some $E_i$ and $E_j$ overlap, and $H$ is a subgraph of $\Gamma$ as $E(H)\subseteq E_4$.

\begin{claim}\label{lm:gamma-sparse}
The graph $\Gamma$ is $(1/4,16\beta t)$-sparse and  $|E_2|\leq 2\beta t v(\Gamma)$.
\end{claim}

\begin{poc}
    %We prove $\Gamma$ is $(1/4,16\beta t)$-sparse by contradiction, and assume that $\Gamma$ is not $(1/4,16\beta t)$-sparse. Then, there exists a pair of vertex subsets $M,N\subseteq V(\Gamma)$ with $|M|,|N|\ge 16\beta t$ such that
%    \begin{equation}\label{eq:e-lb}
%    e(M,N)>\frac{3}{4}|M||N|.
%    \end{equation}
    Let $M,N\subseteq V(\Gamma)$ be a pair of vertex subsets with $|N|\geq|M|\ge 16\beta t$. 
    %In the rest of the proof, we aim to provide an upper bound for $e(M,N)$, which will contradicts to the lower bound. 
    Recall that $\varphi$ is a mapping from $A\cup B=V(\Gamma)$ to $V(G)$ and at most two vertices in $V(\Gamma)$ (at most one from each of $A$ and $B$) can be mapped to the same vertex in $V(G)$. Therefore, we obtain
    \begin{equation}\label{eq:var-size}
    |M|\ge |\varphi(M)|\ge \frac{|M|}{2}\ge 8\beta t\quad\text{and}\quad |N|\ge |\varphi(N)|\ge \frac{|N|}{2}\ge 8\beta t.
    \end{equation}
    For $i\in [4]$, let $E_i(M,N)$ be the set of pairs $(u,v)\in M\times N$ with $uv\in E_i$, and let $e_i(M,N):=|E_i(M,N)|$. Similarly, we use 
 $e(M,N)$ to denote the number of pairs $(u,v)\in M\times N$ with $uv\in E(\Gamma)$.
 Then, 
    \begin{equation}\label{eq:e-sum}
    e(M,N)\le e_1(M,N)+e_2(M,N)+e_3(M,N)+e_4(M,N).
    \end{equation} 
In the following, we bound each $e_i(M,N)$ for $i\in [4]$ to give an upper bound of $e(M,N)$.

    For every vertex $u\in M$, there are at most two vertices $v\in N$ satisfying $\varphi(v)=\varphi(u)$. Hence,
    \begin{equation}\label{eq:e1}
        e_1(M,N)\le 2|M|\le 2|N|.
    \end{equation}
    For $e_4(M,N)$, since there are at most two vertices $u, v$ in $M$ (resp. in $N$) satisfying that $\varphi(u)=\varphi(v)$, each edge in $E_G\left(\varphi(M),\varphi(N)\right)$ is counted at most four times in $E_4(M,N)$. Hence,     
    %Moreover, since $\varphi(m)\in \varphi(M)$ and $\varphi(n)\in \varphi(N)$, we have $(\varphi(m),\varphi(n))\in E(\varphi(M),\varphi(N))$. Given this, consider a bipartite graph with two parts being $E_4(M,N)$ and $E(\varphi(M),\varphi(N))$, where the edge set is the union of all such $\{(m,n),(\varphi(m),\varphi(n))\}$. Trivially, every vertex from $E_4(M,N)$ has degree $1$. On the other hand, since at most two vertices in $\Gamma$ can map to the same vertex of $G$, the degree of any vertex in $E(\varphi(M),\varphi(N))$ is at most $4$. After a double-counting, we have
    
    \begin{equation}\label{eq:e4}
        e_4(M,N)\le 4e_G(\varphi(M),\varphi(N))\le 4\varepsilon |\varphi(M)||\varphi(N)|\le 4\varepsilon|M||N|,
    \end{equation}
    where the second inequality is from the $(1-\varepsilon,\beta t)$-sparseness of $G$ and~(\ref{eq:var-size}), and the last inequality follows from~(\ref{eq:var-size}).

    The following claim gives an upper bound of the number of edges between $M$ and $N$ of type-2 or type-3.
    \begin{claim}\label{cl:e23}
        $e_2(M,N)+e_3(M,N)\le 8\beta t |N|$.
    \end{claim}
    \begin{poc}
        %Let $\Gamma'\subseteq \Gamma$ be the subgraph of $\Gamma$ consisting of all type-2 and type-3 edges, namely $E(\Gamma')=E_2\cup E_3$. Based on $\Gamma'$, we shall construct a digraph $D$ with $V(D)=A\cup B$.
    Construct an auxiliary digraph $D$ with $V(D)=A\cup B$ as follows. 

\begin{enumerate}[label=(\arabic*)]
    \item For every edge $ab\in E_2$ with $a\in A$ and $b\in B$, add an arc $(a,b)$ (oriented from $a$ to $b$).
    
    \item For every $b_1b_2\in E_3$ and $i\in \{1,2\}$, if there are at least $\varepsilon' \gamma t$ $s$-paths $P$ in $\C P_{\varphi(b_i)}$ containing some vertex $v\in V(P)\setminus \{x,\varphi(b_i)\}$ such that $v\varphi(b_{3-i})\in E(G)$, then add an arc $(b_{3-i},b_i)$ (the arcs $(b_1,b_2)$ and $(b_2,b_1)$ may exist simultaneously).
\end{enumerate}

        %For every type-2 edge $ab\in E(\Gamma')$ where $a\in A$ and $b\in B$, we add $(a,b)$ (oriented from $a$ to $b$) into the arc set $A(D)$. Recalling the definition of type-3 edge, for $bb'\in E_3$ where $b,b'\in B$, if the former condition in the definition holds, that is there are at least $\varepsilon' \gamma t$ $s$-paths $P$ in $\C P_{\varphi(b)}$ containing some vertex $v\in V(P)\setminus \{x,\varphi(b)\}$ such that $v\varphi(b')\in E(G)$, we add $(b',b)$ into the arc set $A(D)$. If the latter condition holds, that is there are at least $\varepsilon' \gamma t$ $s$-paths $P$ in $\C P_{\varphi(b')}$ containing some vertex $v\in V(P)\setminus \{x,\varphi(b')\}$ such that $v\varphi(b)\in E(G)$, then we add $(b,b')$ into the arc set $A(D)$. If both conditions hold, we add $(b,b')$ and $(b',b)$ simultaneously. 
        
        \noindent Note that for any vertex subset $Q\subseteq A\cup B$, the number of type-2 and type-3 edges in $\Gamma [Q]$ is at most the number of arcs in $D[Q]$. For every $v\in A\cup B$, let $d^-_D(v)$ be the in-degree of $v$ in $D$. We next show that $d^-_D(v)\leq2\beta t$, and therefore
        \begin{align*}
        e_2(M,N)+e_3(M,N)\le 2e(D[M\cup N])\le 2\sum_{v\in M\cup N} d^-_{D}(v)\le 4(|M|+|N|)\beta t\le 8\beta t |N|,
        \end{align*}
\noindent finishing the proof.

        Note that $d^-_D(a)=0$ for every $a\in A$. Thus it suffices to prove that $d^-_D(b)\leq 2\beta t$ for every $b\in B$. Fix a vertex $b\in B$, set $R:=\{\varphi(u)~|~u \text{~is an in-neighbour of~}b\}$ and $S:=\bigcup_{P\in \C P_{\varphi(b)}} \left(V(P)\setminus\{x,\varphi(b)\}\right)$. Recall that $\C P_{\varphi(b)}$ consists of $\gamma t$ internally vertex-disjoint $s$-paths. Therefore, we have 
        \[
        |S|=(s-1)|\C P_{\varphi(b)}|= (s-1)\gamma t\ge \beta t.
        \]
        On the other hand, by the construction of $D$, for each $w\in R$, there are at least $\varepsilon' \gamma t$ $s$-paths $P$ in $\C P_{\varphi(b)}$ containing some vertex $v\in V(P)\setminus \{x,\varphi(b)\}$ such that $vw\in E(G)$, implying that every vertex $w\in R$ has at least $\varepsilon' \gamma t$ neighbours in $S$. Thus we obtain that 
        \begin{equation}\label{eq:contra}
        e_G(R,S)\ge \varepsilon'\gamma t\cdot|R|=(s-1)\varepsilon\cdot \frac{|S|}{s-1}\cdot |R|=\varepsilon |S||R|.
        \end{equation}
        Since $G$ is $(1-\varepsilon, \beta t)$-sparse, we obtain that $|R|< \beta t$. Consequently, the fact that at most two vertices in $A\cup B$ are mapped to a single vertex in $R$ gives that 
        $d^-_D(b)\le 2|R|<2\beta t$. 
        \end{poc}

By~(\ref{eq:e-sum}),~(\ref{eq:e1}),~(\ref{eq:e4}) and Claim~\ref{cl:e23}, for any subsets $M,N\subseteq V(\Gamma)$ with $|N|\geq|M|\geq 16\beta t$, we get that
\begin{equation}\label{eq:e-rb} 
e(M,N)\le 2|N|+4\varepsilon |M||N|+8\beta t|N|\le \frac{3}{4}|M||N|.
\end{equation}
Thus, $\Gamma$ is $(1/4,16\beta t)$-sparse.

For the second part of the claim, using again the auxiliary  digraph $D$ in the proof of Claim~\ref{cl:e23}, we derive that $|E_2|\leq \sum_{v\in B} d^-_{D}(v)\leq 2\beta t|B|\leq 2\beta t v(\Gamma).$
\end{poc}

\subsection{Finding an induced $C_{2\ell}$ in $G$}

Recall that $H$ is a subgraph of $\Gamma$ and $v(H) = v(\Gamma)$.
Set $m:=2(\ell-s)+1$, and let $C'$ be the constant returned from Theorem~\ref{thm:fnp} with $c'=1/4$ and  $T$ being an $(m+1)$-path. 
Let $H'$ be the graph obtained from $H$ by deleting all edges which are of type-2 in $\Gamma$.
By~\Cref{lm:H-edge} and~\Cref{lm:gamma-sparse}, we have $e(H') \ge \gamma t v(H)/2 - 2\beta t v(H) \ge 16C'\beta t v(H)$, where the last equality follows from that $\gamma\gg \beta,C'$. Then, by~\Cref{lm:gamma-sparse} and Theorem~\ref{thm:fnp}, there is an $(m+1)$-path $Q$ in $H'$ which is induced in $\Gamma$. Further, since $H'$ is bipartite, there is an $(m-1)$-subpath of $Q$ with two endpoints lying in $B$. To see this, we can just delete the last (or the first) two vertices of $Q$ if its two endpoints lie in $B$, and delete the two endpoints of $Q$ if they lie in $A$. 

 The following claim shows that an induced path in $\Gamma$ retained in $H'$ implies an induced path in $G$ with some good properties.   

\begin{claim}\label{lm:2l-2s}
     Let $P=z_1z_2\cdots z_m$ be an $(m-1)$-path in $H'$ with $z_1,z_{m}\in B$, which is induced in $\Gamma$. Let $\varphi(P)$ be the induced subgraph of $G$ on $\varphi(V(P))$. Then the following hold.
    \begin{enumerate}[label=\rm{(\alph*)}]
        \item $\varphi(P)$ is an induced $(m-1)$-path in $G$;
        \item 
        For every $z\in\{z_1,z_m\}$ and every $z_i\neq z$, there are at most $\varepsilon' \gamma t$ $(x,\varphi(z))$-paths in $\C P_{\varphi(z)}$ each of which contains an internal vertex $v$ with $v\varphi(z_i)\in E(G)$;
        \item 
       For every $z\in\{z_1,z_m\}$ and every $z_i\neq z$, there is at most one $(x,\varphi(z))$-path in $\C P_{\varphi(z)}$ containing $\varphi(z_i)$ as an internal vertex;
        
        \item 
        No vertex in $V(\varphi(P))$ is joined to $x$ in $G$.      
    \end{enumerate}  
\end{claim}
\begin{poc}
Note that, for every $i\in [m-1]$, $\varphi(z_i)\varphi(z_{i+1})$ is an edge of $G$ as $z_iz_{i+1}\in E(H')$ is an edge of type-4. So, to prove $(a)$, it suffices to show that $\varphi(z_i)\neq \varphi(z_j)$ and $\varphi(z_i)\varphi(z_j)\notin E(G)$ for every $i,j\in[m]$ with $|i-j|\ge 2$.  Actually, if this does not hold for some $i,j$ with $|i-j|\geq2$, then $z_iz_j$ is either an edge of type-1 in $\Gamma$ (if $\varphi(z_i)=\varphi(z_j)$) or an edge of type-4 in $\Gamma$ (if $\varphi(z_i)\varphi(z_j)\in E(G)$), contradicting that $P$ is an induced path in $\Gamma$.

Suppose that $(b)$ is false for some $z\in\{z_1,z_m\}$ and $z_i\neq z$. Since $z_1,z_m$ lie in $B$, by the construction of $\Gamma$, $z_iz$ is a type-2 or type-3 edge in $\Gamma$. Since $H'$ is bipartite and contains no edge of type-2, we conclude that $z_iz\notin E(P)$; however, $z_iz\in E(\Gamma)$, contradicting that $P$ is an induced path in $\Gamma$.

%For $(c)$, we only prove the case for $\C P_{\varphi(z_1)}$. The statement with $\C P_{\varphi(z_m)}$ will follow from a similar argument. One can notice that for the vertex $\varphi(z_1)$, the conclusion holds trivially, since no path in $\C P_{\varphi(z_1)}$ can contain $\varphi(z_1)$ as an internal vertex. 
Clearly, $(c)$ holds as both $\C P_{\varphi(z_1)}$ and $\C P_{\varphi(z_m)}$ consist of internally vertex-disjoint paths. For $(d)$, 
recalling that $\C F'$ is a set of induced $s$-paths starting from $x$ with $s\ge 3$, and $\varphi(z_i)$ is the $(s-1)$-th or the $s$-th vertex of some $s$-path in $\C F'$, thus $\varphi(z_i)x\notin E(G)$.
\end{poc}

Fix an $(m-1)$-path $P=z_1z_2\cdots z_m$ in $H'$ with $z_1,z_m\in B$ which is induced in $\Gamma$. Then, by~\Cref{lm:2l-2s}, $\varphi(P)=\varphi(z_1)\varphi(z_2)\cdots \varphi(z_{m})$ is an induced $(m-1)$-path in $G$ satisfying properties $(b),(c)$ and $(d)$. The following claim completes the proof of~\Cref{lem: main}.

\begin{figure}
    \centering
    \label{fig:pf-idea}
    \begin{tikzpicture}
    \draw [bend left=30,line width=0.6pt,opacity=0.3] (6.5,-2) to (-1,0);
     \draw [bend left=15,line width=0.6pt,opacity=0.3] (6.5,-2) to (-1,0);
     \draw [bend right=15,line width=0.6pt,opacity=0.3] (6.5,-2) to (-1,0);
    \draw [bend right=30,line width=0.6pt,opacity=0.3] (6.5,2) to (-1,0);
    \draw [bend right=15,line width=0.6pt,opacity=0.3] (6.5,2) to (-1,0);
    \draw [bend right=15,line width=0.6pt,opacity=0.3] (-1,0) to (6.5,2);
 %   \draw[opacity=0.3,line width=0.6pt] (6.5,2) arc (10:20:2);
   \draw[line width=0.78pt,color=blue] (-1,0)--(6.5,-2);
    \draw[line width=0.78pt,color=blue] (-1,0)--(6.5,2);
 %   \draw[decorate, decoration={snake, amplitude=3mm, segment length=7mm}, line width=0.78pt] (6.5,-2) -- (6.5,2);
  % \draw [bend right=15,line width=0.78pt] (6.5,2) to (6.8,2.2);

   \draw[line width=0.78pt, domain=-9.8:9.8,samples=150] plot ({0.5*sin((\x-1.57) r)+6.05},0.2*\x) node[right]{};
    
    \node[scale=1.2] at (-1.3,0) {$x$};
    \node[scale=1.1] at (7.3,0) {$\varphi(P)$};
    \node[scale=1.1] at (7.3,2.1) {$\varphi(z_1)$};
    \node[scale=1.1] at (7.3,-2.1) {$\varphi(z_m)$};
    \node[scale=1] at (3.1,1.37) {$P_1$};
    \node[scale=1] at (3.1,-1.37) {$P_2$};
    \draw [fill=black] (-1,0) circle (1.6pt);
    \draw [fill=black] (6.5,2) circle (1.6pt);
    \draw [fill=black] (6.5,-2) circle (1.6pt);
    \end{tikzpicture}
    \caption{An illustration of finding an induced $C_{2\ell}$}
\end{figure}
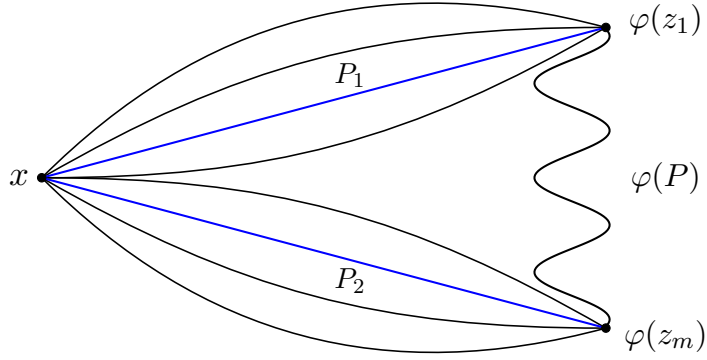

\begin{claim}\label{lm:glue}
    There exist $P_1\in \C P_{\varphi(z_1)}$ and $P_2\in \C P_{\varphi(z_m)}$ such that $P_1\cup P_2\cup \varphi(P)$ is an induced $C_{2\ell}$ in $G$.
\end{claim}

\begin{poc}
For $z\in\{z_1,z_m\}$, let $\C P'_{\varphi(z)}$ be the set obtained from $\C P_{\varphi(z)}$ by removing all the $s$-paths each of which contains some $\varphi(z_i)$ or a neighbour of some $\varphi(z_i)$ as an internal vertex. Then for any $s$-path $P'\in \C P'_{\varphi(z)}$, $P'\cup \varphi(P)$ is an induced $(2\ell-s)$-path in $G$. By $(b)$ and $(c)$ in~\Cref{lm:2l-2s}, the number of $s$-paths removed from $\C P_{\varphi(z)}$ is at most $(m-1)\varepsilon' \gamma t+ (m-1)\le \frac{\gamma t}{2}$, and thus $|\C P'_{\varphi(z)}|\ge \gamma t/2$. 

We next show that there exist $P_1\in \C P'_{\varphi(z_1)}$ and $P_2\in \C P'_{\varphi(z_m)}$ such that no edge in $G$ connects their internal sets (which also implies that their internal sets are disjoint), and then $P_1\cup P_2\cup \varphi(P)$ is an induced $C_{2\ell}$ as desired. 
 
 Let $V_1$ and $V_2$ be the collection of all the internal vertices of paths in $\C P'_{\varphi(z_1)}$ and $\C P'_{\varphi(z_m)}$ respectively.
 Then $|V_1| = (s-1)|\C P'_{\varphi(z_1)}| \ge (s-1) \frac{\gamma t}{2} \ge \beta t$, and $|V_2|= (s-1)|\C P'_{\varphi(z_m)}| \ge \beta t$.
 Suppose to contrary that for any $P_1\in \C P'_{\varphi(z_1)}$ and  $P_2\in \C P'_{\varphi(z_m)}$, there exists an edge connecting their internal sets. Then we have
 \[
 e(V_1,V_2)\ge |\C P'_{\varphi(z_1)}| \cdot |\C P'_{\varphi(z_m)}| = \frac{1}{(s-1)^2}|V_1||V_2| >\varepsilon |V_1||V_2|,
 \]
 which contradicts the $(1-\varepsilon, \beta t)$-sparseness of $G$.
\end{poc}
 
\begin{comment}
 Let $V_1$ (resp. $V_2$) be the union of all internal sets of paths in $\C P'_{\varphi(z_1)}$ (resp. $\C P'_{\varphi(z_m)}$). Then, $|V_i|\ge (s-1)\gamma t/2\ge \beta t$ for $i\in \{1,2\}$. By $(1-\varepsilon,\beta t)$-sparseness of $G$, we get 
 \[
 e(V_1,V_2)\le \varepsilon |V_1||V_2|.
 \]
By pigeonhole principle, there is an interval set $I$ of some path $P_1\in \C P'_{\varphi(z_1)}$ such that the number of pairs $(v_1,v_2)\in I\times V_2$ with $v_1v_2\in E(G)$ is at most
\[
\frac{\varepsilon |V_1||V_2|}{|\C P'_{\varphi(z_1)}|}=(s-1)\varepsilon |V_2|=(s-1)^2\varepsilon |\C P'_{\varphi(z_m)}|.
\]
This implies that there are at most $(s-1)^2\varepsilon |\C P'_{\varphi(z_m)}|$ vertices in $V_2$, denoted by the vertex set $F$, connecting to some vertex in $I$. We remove all the $s$-paths in $\C P'_{\varphi(z_m)}$ containing some $w\in F$ as internal vertex. Since paths in $\C P'_{\varphi(z_m)}$ are internally vertex-disjoint, we remove at most 
$$(s-1)^2\varepsilon |\C P'_{\varphi(z_m)}|<|\C P'_{\varphi(z_m)}|$$ 
paths from $\C P'_{\varphi(z_m)}$. Thus, there is an $s$-path $P_2\in \C P'_{\varphi(z_m)}$ containing no internal vertex connecting to $I$ (which is the internal set of $P_1$). We are done.
\end{comment}

%By the construction of $A$ and $B$, if $\varphi(x)=\varphi(y)$ holds for some $x,y\in A\cup B$, then $x,y$ belong to $A$ and $B$ respectively.

%Construct an auxiliary bipartite graph $H$ with $V(H)=A\cup B$
\section{Concluding remarks} \label{sec:remark}
In this paper, we proved that an induced analogue of Bondy-Simonovits theorem that every $(c,t)$-sparse graph $G$ with at least $C_{c,\ell} t^{1-1/\ell}n^{1+1/\ell}$ edges contains an induced copy of $C_{2\ell}$. It would be interesting to establish, if true, the corresponding supersaturation result.

\begin{prob}\label{Q:1}
    For every $c>0$ and integer $\ell$,  do there exist $C$ and $C'$ such that if an $n$-vertex graph $G$ is $(c,t)$-sparse for some $t$, and has at least $C t^{1-1/\ell}n^{1+1/\ell}$ edges, then $G$ contains at least $C' t^{2\ell -2} n^2 $ induced copies of $C_{2\ell}$?
\end{prob}

 \section*{Acknowledgment}
This work was initiated at the 2nd ECOPRO Student Research Program in the summer of 2024.

\bibliographystyle{abbrv}
\bibliography{cycle}

\end{document}